\documentclass[11pt]{amsart}
\usepackage[T1]{fontenc}
\usepackage{lmodern}
\usepackage[margin=1in]{geometry}
\usepackage{amsmath,amssymb,amsthm,mathtools}
\usepackage{enumitem,microtype}
\usepackage[hidelinks]{hyperref}
\usepackage{url}
\setlist[enumerate]{itemsep=2pt,topsep=5pt}
\newtheorem{theorem}{Theorem}[section]
\newtheorem{lemma}[theorem]{Lemma}
\newtheorem{proposition}[theorem]{Proposition}

\newtheorem{external}[theorem]{Proposition}
\theoremstyle{definition}
\newtheorem{definition}[theorem]{Definition}

\newcommand{\N}{\mathbb N}
\newcommand{\Z}{\mathbb Z}
\newcommand{\Cantor}{2^{\N}}
\newcommand{\Fhat}{\widehat F_{\omega}}
\newcommand{\ES}{E_{S_\infty}}
\newcommand{\onto}{\twoheadrightarrow}
\newcommand{\doi}[1]{\href{https://doi.org/#1}{\nolinkurl{doi:#1}}}
\DeclareMathOperator{\Epi}{Epi}
\DeclareMathOperator{\Bal}{Bal}
\DeclareMathOperator{\Inn}{Inn}
\DeclareMathOperator{\id}{id}
\DeclareMathOperator{\Fact}{Fact}
\DeclareMathOperator{\Homeo}{Homeo}
\DeclareMathOperator{\Min}{Min}

\DeclareMathOperator{\pr}{pr}
\hypersetup{
  pdftitle={The Borel complexity of conjugacy for Cantor minimal systems},
  pdfauthor={Xinan Dai, Wenhao Deng, Yingdong Shi, Tailin Wu, Yuchen Yang}
}
\title[Conjugacy for Cantor minimal systems]{The Borel complexity of conjugacy\\for Cantor minimal systems}

\author{Xinan Dai}
\address{College of Future Information Technology,
Fudan University, Shanghai, China}
\curraddr{Department of Artificial Intelligence,\newline
School of Engineering, Westlake University, Hangzhou, China}
\email{xndai23@m.fudan.edu.cn}

\author{Wenhao Deng}
\address{University of Glasgow, Glasgow, United Kingdom}
\curraddr{Department of Artificial Intelligence,\newline
School of Engineering, Westlake University, Hangzhou, China}
\email{dengwenhao@westlake.edu.cn}

\author{Yingdong Shi}
\address{School of Information Science and Technology,
ShanghaiTech University, Shanghai, China}
\email{shiyd2023@shanghaitech.edu.cn}

\author{Tailin Wu}
\address{Department of Artificial Intelligence,\newline
School of Engineering, Westlake University, Hangzhou, China}
\email{wutailin@westlake.edu.cn}

\author{Yuchen Yang}
\address{Department of Artificial Intelligence,\newline
School of Engineering, Westlake University, Hangzhou, China}
\email{yangyuchen@westlake.edu.cn}

\date{\today}
\begin{document}
\begin{abstract}
We prove that conjugacy of minimal homeomorphisms of the Cantor space is Borel bireducible with isomorphism of countable graphs, answering the Cantor minimal case of a question of Foreman. We obtain the lower bound by encoding countably based profinite groups. Finite quotient homomorphisms are represented by factor maps between a common family of minimal subshifts. Amalgamation makes the resulting inverse limit independent, up to conjugacy, of the quotient presentation. Conversely, finite-stage factorization recovers the group from any conjugacy of these inverse limits.
\end{abstract}
\maketitle

\section{Introduction}

Let $\N=\{0,1,2,\ldots\}$ and let $C=2^{\N}$ be the Cantor space. A homeomorphism of $C$ is \emph{minimal} if every orbit is dense. Two homeomorphisms $T,S$ are \emph{conjugate}, written $T\cong S$, if $hT=Sh$ for some homeomorphism $h$. A minimal system has no nonempty proper closed invariant subset. It is therefore natural to ask whether this restriction makes conjugacy easier to classify. We determine its exact complexity under Borel reducibility.

Write $\Homeo(C)$ for the homeomorphism group, with the topology of uniform convergence of maps and inverses, and $\Min(C)$ for its $G_\delta$ subspace of minimal homeomorphisms. A \emph{standard Borel space} is a measurable space arising from a Polish topology, that is, a separable completely metrizable topology. For equivalence relations $E$ on $X$ and $F$ on $Y$, where $X,Y$ are standard Borel spaces, write $E\leq_B F$ if there is a Borel map $f:X\to Y$ such that
\begin{equation}\label{eq:borel-reduction}
 x\mathrel E x'\quad\Longleftrightarrow\quad f(x)\mathrel F f(x').
\end{equation}
Write $E\sim_B F$ when both reductions hold. A relation is \emph{smooth} if it is Borel reducible to equality on $\mathbb R$. Let $\ES$ denote the universal orbit equivalence relation for Borel actions of the permutation group $S_\infty$ of $\N$, with its topology of pointwise convergence. Thus every such orbit relation reduces to $\ES$. We call $E$ \emph{$\ES$-hard} if $\ES\leq_B E$, and \emph{$\ES$-complete} if $E\sim_B\ES$.

We recall the standard comparison with countable structures, together with the Cantor conjugacy theorem of Camerlo and Gao.
\begin{proposition}\label{prop:countable-universality}
Isomorphism of simple undirected graphs on $\N$ is Borel bireducible with $\ES$ \cite[pp.~492--493]{CamerloGao}. For every countable language, isomorphism of structures with domain $\N$ is Borel reducible to $\ES$ \cite[pp.~492--493]{CamerloGao}. Conjugacy on $\Homeo(C)$ is also Borel bireducible with $\ES$ \cite[Theorem~5(i), p.~511]{CamerloGao}.
\end{proposition}

Restriction to $\Min(C)$ gives the upper bound $\ES$. The issue is whether minimality lowers this degree. Foreman asks precisely whether conjugacy of minimal Cantor homeomorphisms is Borel bireducible with the maximal $S_\infty$ orbit relation \cite[Open Problem~15, p.~78, author version]{Foreman}. We answer this question affirmatively.
\begin{theorem}\label{thm:main}
Conjugacy on $\Min(C)$ is Borel bireducible with $\ES$.
\end{theorem}
Thus minimal Cantor homeomorphisms retain the full classification complexity of countable graphs. The substantial part of the theorem is a Borel construction that realizes this lower bound while preserving minimality.

Vejnar recently obtained the same Borel degree for conjugacy of Devaney
chaotic homeomorphisms of the Cantor space, namely, homeomorphisms with a
dense orbit and a dense set of periodic points \cite[Theorem~3.4]{Vejnar}.
Every minimal Cantor system is infinite and has no periodic points. The two
results therefore locate the same classification degree in disjoint
dynamical regimes.

\subsection{Earlier bounds and the remaining question}

The distinction between a system and a system with a chosen point is central to the history of the problem. A \emph{pointed system} $(C,T,x)$ records a point $x\in C$; a pointed conjugacy must carry it to the distinguished point of the target. Equality of countable sets of reals is represented on $\mathbb R^{\N}$ by
\begin{equation}\label{eq:countable-real-sets}
 (x_n)_n\mathrel{\Delta_{\mathbb R}^{+}}(y_n)_n
 \quad\Longleftrightarrow\quad \{x_n:n\in\N\}=\{y_n:n\in\N\}.
\end{equation}
Kaya obtains an exact classification of pointed systems and a lower bound for the unpointed relation.
\begin{proposition}\label{prop:kaya-background}
Pointed conjugacy on $\Min(C)\times C$ is Borel bireducible with $\Delta_{\mathbb R}^{+}$ \cite[Theorem~1.2, author version]{Kaya}. Moreover, $\Delta_{\mathbb R}^{+}$ is Borel reducible to conjugacy on $\Min(C)$ \cite[Theorem~1.1, author version]{Kaya}.
\end{proposition}
The pointed relation is therefore Borel. Whether the unpointed relation is Borel was explicitly asked by Gao in \cite[Section~1.1.2, Question~1.6, p.~3]{BuzziEtAl}. A subset of a standard Borel space is \emph{analytic} if it is the Borel image of a Polish space. An analytic set $A\subseteq X$ is \emph{complete analytic} if, for every analytic $B\subseteq Y$, where $Y$ is a standard Borel space, there is a Borel map $f:Y\to X$ such that $B=f^{-1}(A)$. In joint work presented in Deka's thesis, Deka, Garc\'ia-Ramos, Kasprzak, Kunde and Kwietniak prove the following.
\begin{proposition}\label{prop:deka-background}
Conjugacy is a complete analytic subset of $\Min(C)^2$ \cite[Corollary~5.7.2, p.~67]{Deka}.
\end{proposition}
This answers Gao's question negatively and determines the descriptive complexity of the set of conjugate pairs. Theorem~\ref{thm:main} determines the finer classification invariant given by the Borel reducibility degree of the equivalence relation. In particular, the passage from pointed to unpointed systems raises the degree from $\Delta_{\mathbb R}^{+}$ to $\ES$.

Two nearby classifications help locate this conclusion. Systems are \emph{orbit equivalent} if a homeomorphism carries the orbits of one system onto those of the other. An \emph{invertible compact metrizable system} is a compact metrizable space equipped with a homeomorphism; the pointed version also specifies a point. For these systems we use the standard Borel coding described in \cite[Section~2, Fact~2.1, author version]{LiPeng}.
\begin{proposition}\label{prop:nearby-background}
Orbit equivalence on $\Min(C)$ is Borel bireducible with $\ES$ \cite[Theorem~5.9, author version]{Melleray}. Pointed conjugacy of invertible minimal compact metrizable systems is not Borel reducible to $\ES$ \cite[Theorem~1.1, author version]{LiPeng}.
\end{proposition}
Consequently, conjugacy and orbit equivalence have the same Borel degree on Cantor minimal systems, although they preserve different dynamical information. The result of Li and Peng also shows the importance of the underlying space: for general compact minimal systems, even pointed conjugacy exceeds classification by countable structures.

\subsection{Symbolic systems and inverse limits}

The symbolic case supplies both an earlier line of investigation and the finite stages of our construction. For a finite discrete alphabet $A$, the left shift on $A^{\Z}$ is given by $(\sigma x)(i)=x(i+1)$. A \emph{subshift} is a nonempty closed subset invariant under $\sigma$ and $\sigma^{-1}$. The Hausdorff metric on these closed subsets supplies the Borel structure; we take a countable disjoint union over alphabets $A=\{0,\ldots,k-1\}$, $k\geq2$. A Borel equivalence relation is \emph{countable} if all its classes are countable, and \emph{universal countable} if every countable Borel equivalence relation reduces to it. Let $E_0$ denote eventual equality on $2^{\N}$:
\begin{equation}\label{eq:eventual-equality}
 x\mathrel{E_0}y\quad\Longleftrightarrow\quad
 (\exists n)(\forall m\geq n)\ x(m)=y(m).
\end{equation}
The following results show how the symbolic problem changes when minimality is imposed.
\begin{proposition}\label{prop:symbolic-background}
Conjugacy of subshifts over finite alphabets is a universal countable Borel equivalence relation \cite[p.~113]{Clemens}. On the subspace of infinite minimal binary subshifts, conjugacy is a countable Borel equivalence relation \cite[p.~10, author version]{GJS}, and $E_0$ is Borel reducible to it \cite[Theorem~1.5.3, p.~10, author version]{GJS}.
\end{proposition}

Toeplitz subshifts provide a useful refinement of this picture. A sequence $x\in A^{\Z}$ is \emph{Toeplitz} if every coordinate is periodic: for each $i\in\Z$ there is $p\geq1$ such that $x(i+kp)=x(i)$ for all $k\in\Z$. Its shift orbit closure is a \emph{Toeplitz subshift}. Every finite block in $x$ then recurs with bounded gaps, so this orbit closure is minimal. We consider infinite Toeplitz subshifts. We use the periodic-part definition of \emph{separated holes} from \cite[Sections~2.4--2.5, author version]{KayaToeplitz}. For each integer $p\geq1$, put
\begin{equation}\label{eq:toeplitz-holes}
 H_p(x)=\{i\in\Z:(\exists k\in\Z)\ x(i+kp)\ne x(i)\},
 \qquad
 \delta_p(x)=\inf\{|i-j|:i,j\in H_p(x),\ i\ne j\},
\end{equation}
with $\inf\varnothing=\infty$. The \emph{$p$-skeleton} of $x$ is obtained by replacing its coordinates in $H_p(x)$ by a fixed blank symbol outside $A$. A positive integer $p$ is an \emph{essential period} if its $p$-skeleton contains a nonblank symbol and has least positive period $p$. Choose essential periods $p_n$ with $p_n\mid p_{n+1}$ such that every essential period divides some $p_n$. The subshift has separated holes if $\delta_{p_n}(x)\to\infty$. This condition is independent of the chosen period sequence and Toeplitz generator. A Borel equivalence relation is \emph{hyperfinite} if it is an increasing union of Borel equivalence relations with finite classes.
\begin{proposition}\label{prop:toeplitz-background}
Conjugacy of infinite binary Toeplitz subshifts is not smooth \cite[Theorem~4.2, p.~11, author version]{Thomas}. On the subspace with separated holes, Sabok and Tsankov prove that, for every Borel probability measure $\mu$ on this subspace, the relation is hyperfinite on a Borel subset of full $\mu$-measure \cite[Theorem~1.1, p.~586]{SabokTsankov}. Kaya strengthens this to hyperfiniteness on the whole subspace \cite[Theorem~1, p.~2, author version]{KayaToeplitz}.
\end{proposition}

Each stage of our construction records a finite group together with quotient homomorphisms. Passing to an inverse limit retains a compatible sequence of these data. This passage allows us to obtain the full $S_\infty$ degree from symbolic systems whose conjugacy relation is countable and Borel.

\subsection{Proof strategy}

We pass through \emph{countably based profinite groups}, that is, inverse limits of sequences of finite groups. The complexity result of Kechris, Nies and Tent \cite[Theorem~4.3, p.~10, author version]{KNT} supplies the source of our reduction; its precise finite-quotient form is stated in Proposition~\ref{ext:knt} and put into our parameter space in Proposition~\ref{prop:profinite-source}. Finite groups are suitable for symbolic encoding because their operations can be recorded by finite tables. The quotient homomorphisms carry the compatibility information needed to reconstruct the group.

A chosen sequence of finite quotients is additional data. Our assignment must depend, up to conjugacy, only on the profinite group. We therefore organize finite quotients together with their homomorphisms and construct common extensions of the resulting finite diagrams. A \emph{factor map} is a continuous surjection intertwining the dynamics. We represent quotient homomorphisms by factor maps between minimal subshifts and preserve their compositions under refinement. Successively meeting the finite extension requirements then makes different quotient presentations give conjugate inverse limits. This use of common extensions and back-and-forth comparison is the projective form of the method developed in \cite{IrwinSolecki,Kubis}; we prove the required extension properties directly.

Recovery imposes a complementary requirement. We must control all factor maps that can arise from an arbitrary conjugacy of the limits. Finite tests ensure that the symbolic maps, after integer shifts, determine homomorphisms of the associated finite groups. A finite-stage factorization then turns a conjugacy and its inverse into compatible maps of quotient systems. Their compositions agree after passage to later stages, and the integer shifts cancel. The resulting inverse maps recover the profinite group.

The finite description of symbolic maps is related to the use of maps between blocks of a common length to study Toeplitz conjugacy by Downarowicz, Kwiatkowski and Lacroix \cite[Section~2]{DKL}. The factor criteria from Deka's thesis, recalled in Propositions~\ref{ext:aligned} and~\ref{ext:tame}, govern the symbolic construction. The categorical treatment of Cantor minimal systems and factor maps by Amini, Elliott and Golestani \cite{AEG} provides another approach to studying the systems together with the maps between them. Here the essential task is to preserve an entire finite diagram while extending it, so that invariance under group isomorphism and recovery from conjugacy hold for the same construction.

Section~\ref{sec:prelim} fixes the parameters and the symbolic criteria. Section~\ref{sec:finite} constructs finite refinements, and Section~\ref{sec:symbolic} realizes them by subshifts. Sections~\ref{sec:amalgamation} and~\ref{sec:recovery} establish invariance and recovery, respectively.
\section{Parameters and symbolic factors}\label{sec:prelim}

\subsection{Profinite groups and Cantor homeomorphisms}\label{sec:parameters}

For profinite groups, $P\cong_{\mathrm{top}}Q$ means that there is a group isomorphism $P\to Q$ that is also a homeomorphism.

Let $\Fhat$ be the free profinite group on a sequence $(x_i)_{i\in\N}$ converging to the identity. Concretely, take the abstract free group on these generators and the kernels of its homomorphisms to finite groups that send all but finitely many $x_i$ to the identity. These kernels form a countable directed family of finite-index normal subgroups and separate points, since finitely generated free groups are residually finite. The resulting topology is Hausdorff and precompact; its completion $\Fhat$ is therefore a compact metrizable group. Every sequence converging to the identity in a profinite group determines a unique continuous homomorphism from $\Fhat$ by sending $x_i$ to its $i$th term.

Every countably based profinite group $P$ is a quotient of $\Fhat$. Indeed, choose a decreasing basis $P=U_0\supseteq U_1\supseteq\cdots$ of open normal subgroups. Concatenate finite generating lifts for the groups $U_n/U_{n+1}$, padding by identities if necessary. This sequence tends to the identity and generates a dense subgroup of $P$, since it generates every finite quotient $P/U_n$. The universal property gives a continuous epimorphism $\Fhat\onto P$.

For a compact metrizable space $Z$, let $K(Z)$ be the hyperspace of its nonempty compact subsets, with the Hausdorff topology. We use
\[
 \mathcal N=\{R\in K(\Fhat):R\text{ is a normal subgroup}\},
 \qquad P_R=\Fhat/R.
\]
The space $\mathcal N$ is closed. For example, if $R_j\to R$ and $x,y\in R$, choose $x_j,y_j\in R_j$ converging to $x,y$. Then $x_jy_j^{-1}\to xy^{-1}$ gives $xy^{-1}\in R$; conjugation is treated in the same way.

Fix a repetition-free enumeration $(M_k)_k$ of the open normal subgroups of $\Fhat$, with multiplication tables for their finite quotients. Enumerating finite-support quotient homomorphisms gives such an enumeration. Equality and inclusion of kernels are decided in the finite image of the product homomorphism; this also computes the induced quotient map $\overline\pi_{l,k}:\Fhat/M_l\onto\Fhat/M_k$ when $M_l\subseteq M_k$. The condition $R\subseteq M_k$ is closed, since $M_k$ is clopen, and then we have the canonical identification
\begin{equation}\label{eq:quotient-presentation}
 P_R/(M_k/R)=\Fhat/M_k.
\end{equation}
Thus finite quotients and their maps have a Borel presentation by natural-number witnesses.

Enumerate the clopen algebra $(U_i)_i$ of $C$. Minimality is expressed by
\begin{equation}\label{eq:minimal-gdelta}
 \Min(C)=\bigcap_{i:U_i\ne\varnothing}\ \bigcup_{n\geq0}
 \left\{T:C=\bigcup_{|j|\leq n}T^j[U_i]\right\}.
\end{equation}
For fixed $i,n$, the finitely many clopen images in this formula are locally constant as functions of $T$. The set in braces is therefore clopen, proving the $G_\delta$ assertion used in the introduction.

We recall the finite-quotient construction of Kechris, Nies and Tent. Fix an odd prime $p$. A simple undirected graph $D$ on $\N$ is
\emph{nice} if it has no triangles or four-cycles and, for every pair
of distinct vertices $u,v$, there is a vertex $w\notin\{u,v\}$ adjacent
to $u$ but not to $v$ \cite[Definition~2.1, p.~2, author version]{ChernikovHempel}.
Write $D\upharpoonright n$ for the induced graph on $\{0,\ldots,n-1\}$. For any such graph, let $Q_n(D)$ be the group on generators
$y_0,\ldots,y_{n-1}$ subject to centrality of all commutators,
exponent $p$, and $[y_i,y_j]=1$ whenever $\{i,j\}$ is an edge of $D$.
Here $[a,b]=a^{-1}b^{-1}ab$, and exponent $p$ means that $g^p=1$ for every group element.

\begin{external}\label{ext:knt}
These groups are finite: each element
has a unique expression consisting of generator powers followed by
powers of $[y_i,y_j]$ for nonedges with $0\leq i<j<n$, with indices in a fixed order and
exponents in $\{0,\ldots,p-1\}$. Thus the multiplication table of
$Q_n(D)$ depends only on $D\upharpoonright n$
\cite[Lemma~4.2 and the proof of Theorem~4.3, pp.~10--11, author version]{KNT}.
The maps $y_n\mapsto1$ and $y_i\mapsto y_i$ for $i<n$ define
surjections $Q_{n+1}(D)\onto Q_n(D)$.
In the notation of \cite[proof of Theorem~4.3, pp.~10--11, author version]{KNT}, if $G(D)$ is the
countable Mekler group and $R_n$ is the normal subgroup generated by
$\{y_i:i\geq n\}$, then $Q_n(D)\cong G(D)/R_n$. Consequently,
$\varprojlim_nQ_n(D)$ is the completion $\widehat G(D)$ used
there.

The following hold.
\begin{enumerate}[label=(\roman*)]
\item\label{ext:knt-nice}
There is a Borel assignment $A\mapsto D(A)$ from graphs on $\N$
to nice graphs such that $A\cong A'$ if and only if
$D(A)\cong D(A')$ \cite[Section~4.2, pp.~9--10, author version]{KNT}.
\item\label{ext:knt-recovery}
For nice graphs $D,D'$, we have
\begin{equation}\label{eq:knt-reconstruction}
D\cong D'\quad\Longleftrightarrow\quad
\varprojlim_n Q_n(D)\cong_{\mathrm{top}}\varprojlim_n Q_n(D').
\end{equation}
This is the reconstruction theorem
\cite[Lemma~4.9, p.~14, author version]{KNT}.
\end{enumerate}
\end{external}

\begin{proposition}\label{prop:profinite-source}
There is a Borel map $A\mapsto R_A$ from graphs on $\N$ to $\mathcal N$ such that
\begin{equation}\label{eq:graph-normal-reduction}
 A\cong A'\quad\Longleftrightarrow\quad
 \Fhat/R_A\cong_{\mathrm{top}}\Fhat/R_{A'}.
\end{equation}
\end{proposition}
\begin{proof}
Take the Borel assignment $A\mapsto D(A)$ from Proposition~\ref{ext:knt}(i). For a nice graph $D$, the assignment $x_i\mapsto y_i$ for $i<n$ and $x_i\mapsto1$ otherwise extends to a continuous epimorphism
$q_{D,n}:\Fhat\onto Q_n(D)$, since its prescribed generator images tend to $1$. The open normal kernels $K_n(D)=\ker q_{D,n}$ decrease with $n$. They depend locally constantly on $D$, by the finite-table assertion of Proposition~\ref{ext:knt}. Put
\begin{equation}\label{eq:graph-kernel-intersection}
 R_A=\bigcap_n K_n(D(A)).
\end{equation}
For every clopen $U\subseteq\Fhat$, nested compactness gives
\begin{equation}\label{eq:kernel-hits}
 R_A\cap U\ne\varnothing
 \quad\Longleftrightarrow\quad
 (\forall n)\ K_n(D(A))\cap U\ne\varnothing.
\end{equation}
The right-hand side is Borel. These clopen hit predicates generate the hyperspace Borel structure, so $A\mapsto R_A$ is Borel. The compatible quotient maps identify $\Fhat/R_A$ with $\varprojlim_nQ_n(D(A))$: the induced map is injective by the definition of $R_A$, and every compatible family of cosets has a common representative by compactness. Now \eqref{eq:knt-reconstruction} gives \eqref{eq:graph-normal-reduction}.
\end{proof}

We will construct a Borel map $R\mapsto T_R$ from $\mathcal N$ to $\Min(C)$ such that
\begin{equation}\label{eq:main}
 P_R\cong_{\mathrm{top}}P_S\quad\Longleftrightarrow\quad T_R\cong T_S.
\end{equation}
Propositions~\ref{prop:countable-universality} and~\ref{prop:profinite-source} then give the lower bound in Theorem~\ref{thm:main}.

\subsection{Construction words and aligned maps}

Positions in finite words are numbered from zero. Let $A$ be a finite alphabet and let $\sigma$ be the left shift on $A^{\Z}$, so $(\sigma x)(t)=x(t+1)$. We use the following conventions for subshifts. If $W$ is a finite set of words of one length, let $\operatorname{Seg}(W)$ consist of all shifts of bi-infinite concatenations of words from $W$.

A \emph{scale} is a strictly increasing sequence $(L_n)$ of positive
integers such that $L_n\mid L_{n+1}$. We use the following form of Deka's
definition \cite[Definition~5.2.1, p.~49]{Deka}. A \emph{construction
sequence} on this scale, beginning at $b\in\N$, is a sequence of finite
sets $W_n\subseteq A^{L_n}$, $n\geq b$, with the following properties for
every $n\geq b$:
\begin{enumerate}[label=(\alph*)]
\item $|W_n|\geq2$;
\item each word in $W_{n+1}$ is a concatenation of $W_n$-words;
\item every ordered pair of $W_n$-words occurs as consecutive blocks in every $W_{n+1}$-word;
\item a $W_n$-word occurring in $uv$, for $u,v\in W_n$, starts at $0$ or $L_n$.
\end{enumerate}
Condition (d) is \emph{unique readability}. Two sequences are compared on
a common tail. Define
\[
 X_W=\bigcap_{n\geq b}\operatorname{Seg}(W_n).
\]
The sets in this intersection are nested, nonempty and compact. Unique readability gives a unique decomposition of each $x\in X_W$ at every level. Every construction word occurs in $X_W$. Indeed, for $w\in W_n$ and every $m\geq n$, the compact set $\operatorname{Seg}(W_m)$ contains a point whose coordinates $0,\ldots,L_n-1$ form $w$, by (c). Intersecting these nested sets with the cylinder specified by $w$ gives a point of $X_W$ with $w$ at a level-$n$ boundary. Every finite word occurring in $X_W$ lies in a pair of sufficiently long construction words; that pair occurs in every word at the next level. Hence the finite word occurs with bounded gaps, and $X_W$ is minimal.

For $n\geq b$, let
$b_n(x)\in\Z/L_n\Z$ be the residue class of the boundaries in the
level-$n$ decomposition. These residues are continuous: a block covering
the origin lies in the window $[-L_n+1,L_n-1]$, and unique readability
determines its starting residue. Set $O=\varprojlim_n\Z/L_n\Z$ and
\begin{equation}\label{eq:rho-convention}
 \rho_{W,n}(x)=-b_n(x)\pmod{L_n}\quad(n\geq b),\qquad
 \rho_W:X_W\longrightarrow O,
\end{equation}
where $\rho_W(x)$ is the unique element of $O$ whose $n$th coordinate is
$\rho_{W,n}(x)$ for every $n\geq b$. Compatibility of the residues makes
this definition possible, and reduction modulo $L_n$ supplies the
coordinates below $b$. The rotation $z\mapsto z+1$ on $O$ is the associated
\emph{odometer}, and $\rho_W(\sigma x)=\rho_W(x)+1$. The image is a
nonempty compact invariant subset of the minimal rotation on $O$, so
$\rho_W$ is onto.

For a factor $F:X_W\to X_V$ on the same length scale, the map
\begin{equation}\label{eq:odometer-offset}
 x\longmapsto \rho_V(Fx)-\rho_W(x)\in O
\end{equation}
is continuous and shift-invariant, hence constant by minimality. Denote
this value by $\tau(F)$ and call it the \emph{odometer offset} of $F$.
The map is \emph{aligned} if $\tau(F)=0$. Under composition these offsets
add. The shift $\sigma^k$ has offset $(k\bmod L_n)_n$; this identifies
$\Z$ with a subgroup of $O$, since $L_n\to\infty$. For a table
$f:W_n\to V_n$, let $f^{[m]}$, for $m\geq n$, apply $f$ to the
$W_n$-blocks of a $W_m$-word. Let $f^*$ apply $f$ to the unique
bi-infinite level-$n$ decomposition, keeping its boundary positions.
Initially $f^*$ takes values in the full target shift.

We use the following aligned-factor criterion of Deka \cite[Lemma~5.2.8, p.~51]{Deka}.
\begin{external}\label{ext:aligned}
For construction sequences $W,V$ on the same scale and $f:W_n\to V_n$, the map $f^*$ is an aligned factor $X_W\to X_V$ if and only if
\begin{equation}\label{eq:external-aligned}
 f^{[m]}[W_m]=V_m\qquad(m\geq n).
\end{equation}
If $f$ is bijective and this condition holds, $f^*$ is a conjugacy.
\end{external}

The second criterion ensures that every factor is aligned after an integer shift. For $3\mid L$, the middle third of a word of length $L$ consists of its letters at positions $L/3,\ldots,2L/3-1$. In a concatenation of length-$L_n$ words, brackets $[\cdot]_j$ denote its $j$th such block, numbered from zero. Deka's second criterion is the following \cite[Lemma~5.2.11, pp.~51--52]{Deka}.
\begin{external}\label{ext:tame}
Let $W,V$ be construction sequences on the same scale. Suppose that, for all sufficiently large $n$:
\begin{enumerate}[label=(\roman*)]
\item distinct words in $V_n$ have distinct middle thirds;
\item writing $\ell=L_{n+1}/L_n$, for every $1\leq c\leq\ell-2$, every table $\psi:W_n^2\to V_n$, every $x,y\in W_{n+1}$ and $z\in V_{n+1}$, some $0\leq j<\ell$ satisfies
\begin{equation}\label{eq:external-cross}
 [z]_j\ne\psi([xy]_{j+c},[xy]_{j+c+1});
\end{equation}
\item $3\mid L_n$ and $L_n\geq6L_{n-1}$.
\end{enumerate}
Then every factor $F:X_W\to X_V$ has the form $F=\sigma^k f^*$ for some $k\in\Z$, some common level $n$ and some table $f:W_n\to V_n$, with $f^*$ aligned.
\end{external}

These criteria separate the two tasks in our word construction: retaining the intended block maps and forcing every factor to arise from a block map. The finite group refinements below address the first task; the separation conditions in Section~\ref{sec:symbolic} address the second.

\section{Finite groups and regular refinements}\label{sec:finite}

We encode a finite quotient homomorphism by a map between finite sets with free group actions. Additional orbit coordinates allow common extensions of finite diagrams; constant fibre sizes make refinement compatible with every retained map.

\subsection{Affine maps and common extensions}
\begin{definition}\label{def:affine}
All finite sets with free group actions in this paper are nonempty. A finite free right $G$-set is written as $I\times G$, using chosen orbit representatives. We call a map $p:I\times G\to J\times H$ \emph{regular affine} if
\begin{equation}\label{eq:affine}
 p(i,g)=(\varphi(i),c_iq(g)),
\end{equation}
where $q:G\onto H$ is an epimorphism, $c_i\in H$, and $\varphi:I\onto J$ has constant positive fibres. We call $q$ the group homomorphism associated with $p$.
\end{definition}
The left placement of $c_i$ gives right equivariance:
\[
 p(i,gr)=p(i,g)q(r)\qquad(i\in I,\ g,r\in G).
\]
For any finite free action with its underlying set coded by natural numbers, we choose the least point of each orbit as its representative. This fixes its coordinates $I\times G$ without additional choices.

\begin{lemma}\label{lem:affine}
Regular affine maps contain identities and are closed under composition. The evaluated table of such a map determines its associated group homomorphism, and this assignment respects identities and composition.
\end{lemma}
\begin{proof}
For $p(i,g)=(\varphi(i),c_iq(g))$ and $r(j,h)=(\psi(j),d_js(h))$, direct evaluation gives
\begin{equation}\label{eq:affine-composition}
 (r\circ p)(i,g)=\bigl(\psi\varphi(i),d_{\varphi(i)}s(c_i)(sq)(g)\bigr).
\end{equation}
If the fibres of $\varphi$ and $\psi$ have sizes $a$ and $b$, those of $\psi\varphi$ have size $ab$. The associated homomorphism is $sq$. Identities are given by $\varphi=\id$, $q=\id$ and $c_i=1$. Finally, evaluation at $(i,1)$ recovers $\varphi(i),c_i$. Writing $\pr_H$ for the group-coordinate projection, we obtain
\begin{equation}\label{eq:recover-q}
 q(g)=c_i^{-1}\pr_H(p(i,g))\qquad(i\in I,\ g\in G).
\end{equation}
Hence equality of evaluated tables implies equality of the associated homomorphisms.
\end{proof}

The positivity assertion in the next lemma is the finite-group lifting theorem of Gasch\"utz \cite[pp.~249--252]{Gaschutz}. We include the counting argument because its constant-fibre conclusion is also needed.

\begin{lemma}\label{lem:lifts}
Let $q:G\onto H$ be an epimorphism of finite groups and $m\geq1$. For a generating tuple $\mathbf h=(h_1,\ldots,h_m)$ of $H$, the integer
\begin{equation}\label{eq:lift-count}
 N_q(\mathbf h)=\left|\left\{\mathbf g\in G^m:q(g_i)=h_i\ (1\leq i\leq m),
 \langle g_1,\ldots,g_m\rangle=G\right\}\right|
\end{equation}
is independent of $\mathbf h$. If $m\geq d(G)$, where $d(G)$ is the least size of a generating set, this integer is positive.
\end{lemma}
\begin{proof}
We induct on $|G|$. The assertion is immediate for the trivial group. Set $K=\ker q$. Every tuple lifting $\mathbf h$ generates a subgroup $D\leq G$ with $DK=G$. For such a subgroup, $q_D=q|_D:D\onto H$ is onto. Partitioning the $|K|^m$ lifts according to the subgroup they generate gives
\begin{equation}\label{eq:lift-recursion}
 |K|^m=\sum_{\substack{D\leq G\\DK=G}}N_{q_D}(\mathbf h),\qquad
 N_q(\mathbf h)=|K|^m-\sum_{\substack{D<G\\DK=G}}N_{q_D}(\mathbf h).
\end{equation}
The induction hypothesis makes each term in the last sum independent of $\mathbf h$, proving the first assertion. If $m\geq d(G)$, choose a generating $m$-tuple of $G$, padding with identities when needed. Its image is a generating tuple of $H$ with a generating lift. The common value of $N_q$ is therefore positive.
\end{proof}

\begin{definition}\label{def:cone}
A finite cone consists of a nonempty finite free right $G_0$-set $U_0$,
its \emph{root}, and, for some $d\in\N$, regular affine maps
$p_j:U_0\to U_j$, $1\leq j\leq d$, to named finite free right
$G_j$-sets. Its \emph{full graph} is
\begin{equation}\label{eq:full-cone-graph}
 \Gamma=\{(u,p_1(u),\ldots,p_d(u)):u\in U_0\}.
\end{equation}
We identify $\Gamma$ with $U_0$ through the root coordinate and transport the right $G_0$-action to it. A map between two cone graphs is regular affine if its map between the roots is regular affine under these identifications.
\end{definition}

For finitely many such graphs $\Gamma_D$ with root groups $G_D$, a \emph{compatible group} consists of a finite group $G_*$ and epimorphisms $q_D:G_*\onto G_D$ such that
\begin{equation}\label{eq:compatible-group-cone}
 q_B=q_fq_D
 \quad\text{for every specified map }f:\Gamma_D\to\Gamma_B,
\end{equation}
where $q_f:G_D\onto G_B$ is the homomorphism associated with $f$. The group $G_*$ acts on each graph through $q_D$.

We use two kinds of finite extensions: a finite product of named cone graphs, or a fibre product
\begin{equation}\label{eq:cone-pullback}
 \Lambda=\Gamma_{D_0}\times_{\Gamma_B}\Gamma_{D_1}
 =\{(z_0,z_1):f_0(z_0)=f_1(z_1)\},
 \qquad f_i:\Gamma_{D_i}\onto\Gamma_B\quad(i=0,1),
\end{equation}
where both $f_i$ are regular affine. The entire base graph and both maps are part of the data. For a product there are no identifications between its named factors; in \eqref{eq:cone-pullback} only the displayed equality is imposed. A compatible group acts coordinatewise on these sets. In the fibre-product case, \eqref{eq:compatible-group-cone} gives $q_B=q_{f_i}q_{D_i}$ for $i=0,1$, so both maps are equivariant for these actions. Further extensions use the resulting cone graphs as their parents.

\begin{lemma}\label{lem:birth}
Let $\Lambda$ be a finite product of cone graphs or a fibre product as in \eqref{eq:cone-pullback}, and let $G_*$ be a compatible finite group. For a finite set $E$ with $|E|\geq2$, the set
\begin{equation}\label{eq:birth}
 W_*=E\times G_*\times\Lambda
\end{equation}
has a free right $G_*$-action and onto regular affine projections to all displayed roots. Every specified equation commutes as an equality of evaluated tables.
\end{lemma}
\begin{proof}
Products of finite nonempty sets have onto constant-fibre projections. If $f_1$ in \eqref{eq:cone-pullback} is $d$-to-one, then $\Lambda\to\Gamma_{D_0}$ is $d$-to-one; the analogous statement holds for the other projection. Composing either projection with its map to $\Gamma_B$ gives an onto constant-fibre map to the base. The root coordinates therefore give onto constant-fibre maps from $\Lambda$ to every displayed root. Compatibility supplies an action of $G_*$ on $\Lambda$ preserving its defining equation. Put
\begin{equation}\label{eq:birth-action}
 (e,g,z)\cdot h=(e,gh,z\cdot h)
 \quad((e,g,z)\in W_*,\ h\in G_*).
\end{equation}
This action is free. A projection to an old root is equivariant for its specified epimorphism $q:G_*\onto G$. In free-action coordinates it therefore has form \eqref{eq:affine}. For a fixed target orbit, each source orbit mapping onto it contributes exactly $|\ker q|$ points to the fibre over each of its points. Thus constant total fibre size implies a constant number of source orbits over each target orbit. The induced map on orbit sets consequently has constant fibres. The pullback equations hold pointwise by construction.
\end{proof}

For word sets of a common length $L$, we realize the new root on a new finite alphabet. The letter at position $0\leq r<L$ of the word named by $(e,g,z)$ records
\begin{equation}\label{eq:birth-alphabet}
 (r,e,g,(w_j(r))_j),
\end{equation}
where $(w_j)_j$ are the old root words recorded by $z$. The phase $r$ makes these words uniquely readable. Coordinate projection realizes each old root word, and the displayed group action is transported to the new word set. For an initial word system, take $\Lambda$ to be a singleton. The action fixes $e$, so every new root has at least two free orbits.

\subsection{Regular refinement and continuation}
For finite nonempty sets $K,I$ with $|I|\mid|K|$, define
\[
 \Bal(K,I)=\{h:K\to I:|h^{-1}\{i\}|=|K|/|I|\text{ for all }i\in I\}.
\]
The next refinement has two roles: balanced maps control the orbit fibres, and generating tuples force every map that respects the refinement to induce one group homomorphism.

Let $F_m$ be the abstract free group on $z_1,\ldots,z_m$. Write $\Epi(F_m,G)$ for the set of epimorphisms, identified with the generating $m$-tuples of $G$. For a fixed finite family of sets $I\times G$ with free right actions, choose $m\geq\max\{1,d(G):G\text{ occurs in the family}\}$ and a finite nonempty set $K$ whose cardinality is divisible by every $|I|$ in the family. Set
\begin{equation}\label{eq:label-set}
 \mathcal L_{I,G}=\Bal(K,I)\times\Epi(F_m,G)\times G^K.
\end{equation}
For \eqref{eq:affine}, define
\begin{equation}\label{eq:label-map}
 \widehat p(h,\lambda,a)=\bigl(\varphi h,q\lambda,
 (c_{h(k)}q(a_k))_{k\in K}\bigr).
\end{equation}
The label set carries the free right action
\begin{equation}\label{eq:label-action}
 (h,\lambda,(a_k)_k)\cdot r
 =\bigl(h,\Inn(r^{-1})\lambda,(a_kr)_k\bigr),
 \qquad \Inn(r^{-1})(g)=r^{-1}gr\quad(r,g\in G).
\end{equation}
For $k\in K$ and $w\in F_m$, define the evaluation coordinates by
\begin{equation}\label{eq:realization}
 H_{(h,\lambda,a)}^{k,w}=(h(k),a_k\lambda(w)),\qquad k\in K,\ w\in F_m.
\end{equation}
We will restrict this evaluation map to $K\times\Omega$ for a finite set $\Omega\subseteq F_m$.

\begin{lemma}\label{lem:cover}
For the fixed finite family, one can compute a finite $\Omega\subseteq F_m$ containing the identity such that:
\begin{enumerate}
\item every regular affine map has an onto, constant-fibre label map \eqref{eq:label-map};
\item each label's realization map $K\times\Omega\to I\times G$ is onto;
\item if a set map $f:I\times G\to J\times H$ sends the realization of a source label to that of a target label at every coordinate, then $f$ is regular affine;
\item the realizations distinguish distinct labels of each set with its given free action.
\end{enumerate}
Moreover, for every regular affine map
$p:I\times G\to J\times H$ with associated epimorphism $q$, every source
label $\ell\in\mathcal L_{I,G}$, every $r\in G$, and every
$(k,w)\in K\times\Omega$,
\begin{equation}\label{eq:realization-identities}
 H_{\widehat p(\ell)}^{k,w}=p(H_\ell^{k,w}),\qquad
 H_{\ell\cdot r}^{k,w}=H_\ell^{k,w}\cdot r,\qquad
 \widehat p(\ell\cdot r)=\widehat p(\ell)\cdot q(r).
\end{equation}
\end{lemma}
\begin{proof}
If $\varphi$ is $d$-to-one and $N=|K|/|I|$, a balanced target map has
\[
 \left(\frac{(dN)!}{(N!)^d}\right)^{|J|}
\]
balanced lifts. The generating-tuple lift count is positive and constant by Lemma~\ref{lem:lifts}, and the translation table has $|\ker q|^{|K|}$ lifts. Multiplying the three counts proves the first assertion. Substitution in \eqref{eq:label-map}--\eqref{eq:realization} proves \eqref{eq:realization-identities}; in particular $q\Inn(r^{-1})=\Inn(q(r)^{-1})q$. Formula~\eqref{eq:label-action} is a right action, and is free since $a_kr=a_k$ implies $r=1$.

For each epimorphism $\lambda:F_m\onto G$ in the fixed family and each $g\in G$, include a word $w_{\lambda,g}$ with $\lambda(w_{\lambda,g})=g$. Also consider all pairs of functions on $F_m$ of the forms
\begin{equation}\label{eq:test-functions}
 w\longmapsto f(i,a\lambda(w)),\qquad
 w\longmapsto (j,b\mu(w)),
\end{equation}
where $f$ is any set map between the given sets with free actions and all other entries range over their finite possibilities. Include one disagreement word for each unequal pair. This is a finite computable procedure: generate the finite subgroup
\[
 \langle(\lambda(z_s),\mu(z_s)):1\leq s\leq m\rangle\leq G\times H
\]
by breadth-first search, and test $f(i,ag)=(j,bh)$ on each element $(g,h)$ of this subgroup. Equality on all of them is equality of the two functions; otherwise the search supplies a disagreement word. Include the identity word and, if necessary, witnesses distinguishing two evaluation maps on $K\times F_m$. Such witnesses are obtained by the same finite procedure, or by applying it with $f=\id$.

For a given label $(h,\lambda,a)$ and point $(i,g)$, choose $k$ with $h(k)=i$ and then $w_{\lambda,a_k^{-1}g}$. This proves surjectivity of the evaluation map. For a source label $(h,\lambda,a)$ and a target label $(h',\mu,b)$, where $b=(b_k)_{k\in K}$, suppose that
\begin{equation}\label{eq:label-rigidity}
 f(h(k),a_k\lambda(w))=(h'(k),b_k\mu(w))
\end{equation}
for every $k\in K,w\in\Omega$. Our choice of witnesses extends this identity to all $w\in F_m$. Since $\lambda$ is onto, the target orbit coordinate is constant on each source orbit. Taking $w\in\ker\lambda$ and comparing with $w=1$ gives $\ker\lambda\subseteq\ker\mu$. Thus $\mu=q\lambda$ for a single epimorphism $q:G\onto H$. Equation~\eqref{eq:label-rigidity} then becomes
\[
 f(h(k),g)=(h'(k),b_kq(a_k)^{-1}q(g))
 \qquad(k\in K,\ g\in G).
\]
Repeated values of $h(k)$ give a well-defined map on orbit sets $\varphi$
and translations $c_i$. Balance of $h,h'$ implies, for every $j\in J$,
\[
 |\varphi^{-1}\{j\}|\,\frac{|K|}{|I|}=\frac{|K|}{|J|},
\]
so $\varphi$ is onto with constant fibres. This proves the third assertion. Finally the evaluation map on $K\times F_m$ recovers $h$ from first coordinates, every $a_k$ at $w=1$, and then $\lambda$ by cancellation. Our finite witnesses preserve these distinctions on $\Omega$, proving the fourth assertion.
\end{proof}

\begin{lemma}\label{lem:continuation}
The label operation respects identities and composition. After labels are realized as distinct words using a common sequence of realization coordinates, it gives a uniquely specified continuation of every retained table. Continuation preserves regularity, surjectivity, the associated group homomorphism, and all equalities between composites of retained tables.
\end{lemma}
\begin{proof}
For composable $p,r$ as in \eqref{eq:affine-composition}, fix a source
label $(\chi,\lambda,a)\in\mathcal L_{I,G}$ and $k\in K$. The $k$th
translation coordinate of $\widehat r\widehat p(\chi,\lambda,a)$ is
\[
 d_{\varphi(\chi(k))}\,s(c_{\chi(k)})\,(sq)(a_k),
\]
which is the translation coordinate of $\widehat{rp}$. The balanced-map and generating-tuple coordinates also agree. Thus $\widehat{rp}=\widehat r\widehat p$, and $\widehat{\id}=\id$. Lemma~\ref{lem:cover} and \eqref{eq:realization-identities} give the remaining assertions. Passing through a label-to-word bijection preserves these identities.
\end{proof}

At each level, the current cone graph is the graph of the continued maps. We form new products and fibre products from these current graphs; Lemma~\ref{lem:continuation} preserves the resulting projections and equations thereafter.

\section{A common family of subshifts}\label{sec:symbolic}

We now choose words on a common length scale so that the aligned factors between the resulting subshifts are precisely the maps induced by regular affine tables.

\subsection{The simultaneous construction}

Fix effective encodings of the finite groups, sets, maps, words, and
derivations used below. A word system has an index $u$, an initial level
$b(u)$, a finite group $G_u$, a finite alphabet $A_u$, and word sets
\[
 W_{u,n}\subseteq A_u^{L_n}\qquad(n\geq b(u))
\]
with specified free right $G_u$-actions. Set $L_0=3$ and require
$L_n\mid L_{n+1}$ and $L_{n+1}\geq6L_n$. Thus $3\mid L_n$ for all
$n$. A system is \emph{active} at level $n$ if $b(u)\leq n$. For
active $u,v$, write $\mathcal A_n(u,v)$ for all regular affine maps
$W_{u,n}\to W_{v,n}$, with the orbit representatives fixed earlier.

A finite condition graph is the full graph of a cone as in Definition~\ref{def:cone}, with construction word sets as its root and targets. Its construction record consists of an initial
construction and subsequent continuation steps. At any later level,
its current graph is the graph of its continued root-to-coordinate
tables. The continuations of Lemma~\ref{lem:continuation} preserve
their associated group homomorphisms and commuting equations.

An extension request specifies either a seed $(G,e)$, where $G$ is a finite group and $e\geq2$ is the prescribed number of orbits in the initial free $G$-set, or one of the two extensions of Section~\ref{sec:finite}: a finite product of named condition graphs, or a single fibre product over a named condition graph with two specified regular affine maps. A product or fibre-product request also specifies a compatible finite group and a set $E$ with $|E|\geq2$, as in Lemma~\ref{lem:birth}. Each fibre-product request includes its full base graph and both maps; repeated named copies remain separate apart from the specified equality. At execution, the parent root-to-coordinate maps and, in the fibre-product case, the two specified maps are first continued to the current level. We then form the product or fibre product of the current graphs and apply Lemma~\ref{lem:birth}. The construction record stores the new finite set and its projection tables. Subsequent requests use the resulting cone graphs and their continued maps.

Assign code $0$ to the trivial-group seed with two orbit representatives.
At level $n$, repeatedly execute the least unexecuted valid request of
code at most $n$. A request is considered only after every level named
in its parent conditions, tables, and continuation records is at most
$n$. Validity then requires that its parents have been introduced, its
tables have the indicated types, and its continued diagram has the
indicated types and equations, including the compatibility equations for
the specified group epimorphisms. These are finite tests. For a seed
$(G,e)$, define $w_{i,g}(r)=(r,i,g)$ for $0\leq i<e$, $g\in G$, and
$0\leq r<L_n$; the seed word set is $\{w_{i,g}:0\leq i<e,\ g\in G\}$,
with $w_{i,g}\cdot h=w_{i,gh}$ for $h\in G$.
Other requests use Lemma~\ref{lem:birth} at length $L_n$. The alphabet
is fixed when the system is introduced. The position coordinate makes
these initial words uniquely readable.

There are finitely many codes below the cutoff, so this procedure
terminates at each level. It is also fair: once the parents and named
levels of a valid request are available, continuation preserves its
typing and equations, while its compatible group and the specified
epimorphisms are unchanged. The request is
executed once its code is below the cutoff and these prerequisites have
been met. Induction on its finite derivation proves this for every
valid request.

After executing these requests, apply Lemma~\ref{lem:cover} to all
active word sets and all their regular affine maps, with common choices
of $K,m,\Omega$. Put $\mathcal H_n=K\times\Omega$. For a refinement
label $\xi$ of system $u$, let
\[
 X_{u,\xi}:\mathcal H_n\longrightarrow W_{u,n}
\]
be its evaluation map. These maps are onto and distinguish the labels.
For $\ell\geq6$ and $H_0,\ldots,H_{\ell-1}\in\mathcal H_n$, form
the candidate words
\begin{equation}\label{eq:next-word}
 w_{u,\xi}=X_{u,\xi}(H_0)\cdots X_{u,\xi}(H_{\ell-1}).
\end{equation}
They are indexed by labels during the search. Write $[x]_j$ for the
$j$th lower block of a candidate $x$, and use indices
$0,\ldots,2\ell-1$ for a concatenation $xy$.

\begin{definition}\label{def:certificate}
A sample $(H_0,\ldots,H_{\ell-1})$ is accepted if the following
conditions hold for all active systems and all the indicated finite
choices.
\begin{enumerate}
\item Every coordinate of $\mathcal H_n$ occurs in the sample.
\item Every ordered pair of lower words occurs in every candidate
upper word of the same system.
\item Within each system, candidates indexed by distinct labels have distinct middle thirds.
\item For every active system $u$ and all candidates $x,y,z$ of $u$,
the word $z$ does not occur at a nonboundary position in $xy$.
\item For every $u,v$, every map $\psi:W_{u,n}^2\to W_{v,n}$,
all source candidates $x,y$, every target candidate $z$, and every
$1\leq c\leq\ell-2$, some $0\leq j<\ell$ satisfies
\begin{equation}\label{eq:cross-offset}
 [z]_j\ne\psi([xy]_{j+c},[xy]_{j+c+1}).
\end{equation}
\end{enumerate}
\end{definition}

We use the following elementary estimate for events depending on
overlapping sets of independent coordinates.

\begin{lemma}\label{lem:independent-supports}
Let $(H_i)_i$ be independent random variables, and let $E_1,\ldots,E_M$
be events with $\Pr(E_j)\leq1-\varepsilon$, where
$0<\varepsilon<1$. Suppose $E_j$ depends only on coordinates in a set
$S_j$, where $s,d$ are positive integers, $|S_j|\leq s$, and each
coordinate belongs to at most $d$
of these sets. Then
\[
 \Pr\left(\bigcap_{j=1}^M E_j\right)
 \leq(1-\varepsilon)^{M/(1+s(d-1))}.
\]
\end{lemma}

\begin{proof}
The intersection graph of the sets $S_j$ has maximum degree at most
$s(d-1)$. Greedy selection gives at least $M/(1+s(d-1))$ pairwise
disjoint supports. Their events are independent, and their intersection
contains the intersection of all the $E_j$.
\end{proof}

\begin{lemma}\label{lem:termination}
An accepted sample exists at every level. Searching in increasing
length, and then in lexicographic order, computes one from the finite
level data.
\end{lemma}

\begin{proof}
Fix $n$ and a prospective value of $\ell$, put $r=|\mathcal H_n|$, and
choose $(H_i)_{0\leq i<\ell}$ independently and uniformly from
$\mathcal H_n$. All word sets, labels, and finite
maps are fixed independently of $\ell$. By surjectivity of every evaluation
map,
\begin{equation}\label{eq:positive-support}
 p_n=\min \Pr(X_{u,\xi}(H_0)=w)>0,
\end{equation}
where the minimum ranges over active systems $u$, refinement labels $\xi$
of $u$, and words $w\in W_{u,n}$. This is a finite nonempty family.
The variables $H_i$ are identically distributed. Since each word set has at
least two elements, no point probability of $X_{u,\xi}(H_i)$ exceeds
$1-p_n$ for any $0\leq i<\ell$.

For condition 1, the probability of omitting one coordinate is
$(1-r^{-1})^\ell$. For condition 2, inspect the disjoint sample pairs
$(H_0,H_1),(H_2,H_3),\ldots$. A specified ordered pair in one label
coordinate has probability at least $p_n^2$ at each trial, so its
absence probability is at most
$(1-p_n^2)^{\lfloor\ell/2\rfloor}$. For condition 3, distinct labels
have different evaluation maps by Lemma~\ref{lem:cover}, and therefore
disagree at a uniform sample coordinate with probability at least
$r^{-1}$. Their middle thirds contain at least
$\lfloor\ell/3\rfloor-2$ complete lower blocks. For $\ell\geq9$,
the probability of agreement of these middle thirds is consequently at
most $(1-r^{-1})^{\lfloor\ell/3\rfloor-2}$.

For conditions 4 and 5, fix the system indices, tables, candidate
labels, and an offset $c$. In condition 4, let $\xi,\eta,\zeta$ be the
labels of $x,y,z$, respectively. An occurrence of $z$ in $xy$ at the
offset $cL_n$, where $1\leq c\leq\ell-1$, would require
\[
 X_{u,\zeta}(H_j)=
 \begin{cases}
 X_{u,\xi}(H_{j+c}),&j+c<\ell,\\
 X_{u,\eta}(H_{j+c-\ell}),&j+c\geq\ell,
 \end{cases}
 \qquad 0\leq j<\ell.
\]
Thus equality at lower block $j$ depends on the two sample indices
$j$ and $j+c\pmod\ell$. They are distinct,
and each index occurs in at most two of these supports, including when
$2c=0\pmod\ell$. Conditional on the source coordinate, the target
equality has probability at most $1-p_n$ by
\eqref{eq:positive-support}. Lemma~\ref{lem:independent-supports},
with $(M,s,d)=(\ell,2,2)$, bounds the probability of the proposed
occurrence by $(1-p_n)^{\ell/3}$. If the symbol offset is not divisible
by $L_n$, its first lower block would occur at a nonboundary position
in a pair of lower words, contradicting the readability already
established at level $n$.

For condition 5, the source label is likewise chosen according to
which side of the seam in $xy$ contains the relevant block. The
forbidden equality at $j$ has support
\begin{equation}\label{eq:raw-support}
 \{j,\ j+c\pmod\ell,\ j+c+1\pmod\ell\}.
\end{equation}
These three indices are distinct for $1\leq c\leq\ell-2$, and
each index belongs to at most three supports. After conditioning on
the two source coordinates, the independent target coordinate again
bounds equality by $1-p_n$. Applying
Lemma~\ref{lem:independent-supports} with
$(M,s,d)=(\ell,3,3)$ gives $(1-p_n)^{\ell/7}$.

The numbers of coordinates, label pairs, word pairs, and finite tables
in these tests are fixed at level $n$. Only the offset contributes a
factor proportional to $\ell$. A union bound therefore gives constants
$C_n<\infty$ and $\varepsilon_n>0$, independent of $\ell$, such that
the total failure probability, for all sufficiently large $\ell$, is
at most
\[
 C_n\ell(1-\varepsilon_n)^{\ell/7}.
\]
For example, take $\varepsilon_n$ no larger than
$\min\{r^{-1},p_n^2\}$ and absorb the bounded loss in the middle-third
exponent into $C_n$, using
$\lfloor\ell/3\rfloor-2\geq\ell/7-3$. This bound tends to zero. Hence an accepted sample
exists. Every acceptance condition is finite and decidable, so the
stated search terminates.
\end{proof}

Use the first accepted sample to define $L_{n+1}=\ell L_n$ and
$W_{u,n+1}$ by \eqref{eq:next-word}. Condition 1 and separation of
the evaluation maps make the label-to-word map injective. Transport
the free label action and the label maps of Lemma~\ref{lem:cover}
through this bijection. The evaluation identities identify the resulting
action with the blockwise action on lower words, and
Lemma~\ref{lem:continuation} gives the continued regular affine maps
with all their composition identities.

For a seed, $e\geq2$ gives at least two group orbits in its initial word
set; for every other birth, $|E|\geq2$ does so. In a refinement, the
balanced map $h$ is fixed by the
group action. An orbit set with at least two points admits at least
two balanced maps, so every later word set again has at least two
orbits. This completes the induction underlying the construction and
the point-probability bound used above.

The same coordinate coverage also excludes nonregular block maps,
without a further probabilistic condition. Put $\ell=L_{n+1}/L_n$.
For any finite table $f:W_{u,n}\to W_{v,n}$, write
$f^{\langle\ell\rangle}$ for its
application to each of $\ell$ consecutive lower blocks. Then
\begin{equation}\label{eq:exclude-nonregular}
 f^{\langle\ell\rangle}[W_{u,n+1}]\subseteq W_{v,n+1}
 \quad\Longrightarrow\quad f\in\mathcal A_n(u,v).
\end{equation}
Indeed, choose any source label $\xi$ and a target label $\eta$ naming
the image of its word. Equality of these words gives
$f(X_{u,\xi}(H_i))=X_{v,\eta}(H_i)$ at every sample position.
Condition 1 extends this equality to every $H\in\mathcal H_n$.
The rigidity assertion of Lemma~\ref{lem:cover} implies that $f$ is
regular affine.

\subsection{Subshifts and their factor maps}

Each pair of word sequences satisfies the hypotheses of Propositions~\ref{ext:aligned} and~\ref{ext:tame} on its common tail. Indeed, concatenation and the persistence of at least two words give (a)--(b) of the construction-sequence definition. Conditions 2 and 4 give ordered-pair coverage and unique readability, while conditions 3 and 5 give middle-third separation and the required two-block exclusion. The length conditions hold by construction.

For an index $u$, let $X_u=X_{W_u}\subseteq A_u^{\Z}$ be the construction subshift of Section~\ref{sec:prelim}, beginning at level $b(u)$. It has the continuous boundary map
\begin{equation}\label{eq:odometer}
 \rho_u:X_u\onto O,\qquad \rho_u(\sigma x)=\rho_u(x)+1,
\end{equation}
with the sign convention of \eqref{eq:rho-convention}.

For a table $f:W_{u,n}\to W_{v,n}$, use the notation $f^{[m]}$ and
$f^*$ from the preliminaries. Its block extension $f^{[m]}$ is initially
a map $W_{u,m}\to A_v^{L_m}$; when $f$ is regular affine, this block extension is exactly its continued table at level $m$.
Write $\Fact(X,Y)$ for the continuous onto maps from $X$ to $Y$
that commute with their shifts.

\begin{theorem}\label{thm:factors}
For all indices $u,v$,
\[
 \Fact(X_u,X_v)=
 \{\sigma^k p^*: k\in\Z,\ n\geq\max\{b(u),b(v)\},\quad
                    p\in\mathcal A_n(u,v)\}.
\]
The integer shift is unique, and the regular affine table is unique
after continuation to a common level. Its evaluated table determines
its associated finite-group epimorphism.
\end{theorem}

\begin{proof}
For a factor $F:X_u\to X_v$, Proposition~\ref{ext:tame} gives $F=\sigma^k f^*$, where $f$ is
a table at a common level $n$ and $f^*$ is aligned. Proposition~\ref{ext:aligned} gives
\begin{equation}\label{eq:all-later-images}
 f^{[m]}[W_{u,m}]=W_{v,m}\qquad(m\geq n).
\end{equation}
Since $\ell=L_{n+1}/L_n$ and
$f^{[n+1]}=f^{\langle\ell\rangle}$, equation
\eqref{eq:all-later-images} at $m=n+1$ and implication
\eqref{eq:exclude-nonregular} show that $f\in\mathcal A_n(u,v)$.

Conversely, a map $f\in\mathcal A_n(u,v)$ has an onto label lift
with the evaluation identity. Its block extension therefore maps
$W_{u,n+1}$ onto $W_{v,n+1}$. Iterating gives
\eqref{eq:all-later-images} at every later level, so Proposition~\ref{ext:aligned} makes $f^*$ an aligned factor.

An aligned factor has zero odometer offset; a shift by $k$ adds
$(k\bmod L_n)_n$. Two representations of the same factor therefore
have the same integer shift, because $L_n\to\infty$. Continue their
remaining tables to a common level. Every source word occurs at a
boundary in some point of $X_u$, and the target decomposition is
unique. Equality of the aligned factors gives equality of the two
tables on every source word. Lemma~\ref{lem:affine} recovers their
associated group epimorphism.
\end{proof}

\section{Amalgamation and the Borel construction}\label{sec:amalgamation}

The word systems encode finite quotient maps. We now organize their
finite diagrams over $P_R$ and select a sequence that absorbs every
compatible finite extension. This removes the dependence on a particular quotient
presentation: an isomorphism of the profinite groups will induce a conjugacy of
the resulting inverse limits.

\subsection{Finite diagrams with quotient witnesses}

We now specify the finite codes for the cones used above. A \emph{condition} is a finite code
\begin{equation}\label{eq:condition-code}
 a=\bigl(n(a),r(a),(s_t)_{t<d(a)},(p_t)_{t<d(a)},D_a\bigr),
 \qquad p_t:W_{r(a),n(a)}\onto W_{s_t,n(a)}.
\end{equation}
The indices $r(a),s_t$ specify word systems active at level $n(a)$; we
call $r(a)$ the root and the $d(a)$ entries $s_t$ the slots. Each $p_t$ is regular
affine, and one slot is the root with its identity map. The finite tree
$D_a$ records the construction. A creation node records a seed or an
application of Lemma~\ref{lem:birth}, with the parent diagrams and
projection tables at its execution level. A continuation node records
an earlier condition, a later level, and the continued maps and equations
given by Lemma~\ref{lem:continuation}. At level $m\geq n(a)$ we use the
current graph
\begin{equation}\label{eq:condition-graph}
 \Omega_a(m)=
 \left\{\bigl(p_t^{\uparrow m}(w)\bigr)_{t<d(a)}:
                 w\in W_{r(a),m}\right\},
\end{equation}
where $p_t^{\uparrow m}$ denotes continuation. The root coordinate
identifies this graph with $W_{r(a),m}$.

Requests involving these conditions are executed by the schedule of Section~\ref{sec:symbolic}. At execution, we continue all parent maps to the current level and form the specified product or pullback from the current graphs \eqref{eq:condition-graph}. The construction record includes this finite set and its projections; later nodes record their continuations. A creation node must agree with a request actually executed by the schedule, including its word-system index and projection tables. A continuation node must agree with the tables computed by the common refinement. Reconstructing the schedule up to the largest referenced level and checking these finite data decides validity.

A \emph{root witness} for a condition $a$ over $P_R=\Fhat/R$ is a pair $(k,\beta)$ with
$R\subseteq M_k$ and an isomorphism of finite groups
$\beta:\Fhat/M_k\to G_{r(a)}$, specified by its table. A condition equipped with this pair is
\emph{witnessed}; equality means equality of all coded fields. We fix
the trivial-group seed as a starting object.

For witnessed conditions $a,b$ with root witnesses
$(k_a,\beta_a)$ and $(k_b,\beta_b)$, an arrow $e:a\to b$ is a regular
affine table
\[
 p_e:W_{r(b),n(b)}\onto W_{r(a),n(b)},\qquad n(a)\leq n(b),
\]
whose associated group epimorphism $q_e:G_{r(b)}\onto G_{r(a)}$ satisfies
\begin{equation}\label{eq:witness-compatibility}
 M_{k_b}\subseteq M_{k_a},\qquad
 \beta_a\overline\pi_{k_b,k_a}=q_e\beta_b.
\end{equation}
It induces the slot maps $p_t^{\uparrow n(b)}p_e$ of $a$. For
$e:a\to b$ and $f:b\to c$, define
\begin{equation}\label{eq:arrow-composition}
 p_{f\circ e}=p_e^{\uparrow n(c)}p_f.
\end{equation}
Lemmas~\ref{lem:affine} and~\ref{lem:continuation} give associativity,
identities, and preservation of \eqref{eq:witness-compatibility}. The arrow $e:a\to b$ represents the aligned factor $p_e^*:X_{r(b)}\to X_{r(a)}$. The evaluation identities for continuation give
\begin{equation}\label{eq:contravariant-factors}
 (p_{f\circ e})^*=p_e^*\circ p_f^*.
\end{equation} All category and
commuting-diagram tests are finite except $R\subseteq M_k$, which is
closed. They are therefore Borel uniformly in $R$.

Joint embedding means that any two objects admit arrows to a common
object. Amalgamation means that a span $a\to b_0,a\to b_1$ extends to
a commuting square. The root kernel of a witnessed condition is the
open normal subgroup $M_k/R$ represented by its witness. We will use
these properties both to compare finite diagrams and to make the
selected kernels cofinal among all open normal subgroups.

\begin{lemma}\label{lem:amalgamation}
The witnessed category has joint embedding and amalgamation. Given an
object and an open normal subgroup of $P_R$, there is an arrow to an
object whose root kernel is contained in that subgroup.
\end{lemma}

\begin{proof}
For objects with root kernels $M_{k_0}/R,M_{k_1}/R$, take the new root
group $\Fhat/(M_{k_0}\cap M_{k_1})$ and request a product construction.
At its execution level, apply Lemma~\ref{lem:birth} to the two current
graphs. The resulting projections satisfy
\eqref{eq:witness-compatibility} with the canonical quotient witness,
so they give a joint embedding.

For a span $e_i:a\to b_i$, retain two named copies. At the execution level $m$, write $\gamma_c^m:W_{r(c),m}\to\Omega_c(m)$ for the root-coordinate identification, for $c=a,b_0,b_1$. Define
\begin{equation}\label{eq:current-graph-maps}
 \Gamma_i=\gamma_a^m\,p_{e_i}^{\uparrow m}\,(\gamma_{b_i}^m)^{-1}:
 \Omega_{b_i}(m)\onto\Omega_a(m),\qquad i=0,1.
\end{equation}
These maps are well-defined and have constant positive fibres because $p_{e_i}^{\uparrow m}$ is regular affine and the root-coordinate identifications are bijections. Form their fibre product
\begin{equation}\label{eq:current-pullback}
 \Lambda=\{(z_0,z_1)\in\Omega_{b_0}(m)\times\Omega_{b_1}(m):
                          \Gamma_0(z_0)=\Gamma_1(z_1)\}.
\end{equation}
Both pullback projections consequently have constant fibres. The group
$\Fhat/(M_{k_{b_0}}\cap M_{k_{b_1}})$ maps compatibly to the two root
groups by \eqref{eq:witness-compatibility}. Lemma~\ref{lem:birth}
supplies a free root set and regular projections $v_i$ satisfying
\[
 p_{e_0}^{\uparrow m}v_0=p_{e_1}^{\uparrow m}v_1.
\]
The quotient witness makes these witnessed arrows. Their composites
agree, and continuation preserves this table equality. This also treats
unequal parallel arrows with $b_0=b_1$.

Finally, given root kernel $M_l/R$ and an open normal subgroup $M_k/R$,
joint embedding with a witnessed seed for $\Fhat/M_k$ gives a root
kernel $(M_l\cap M_k)/R$.
\end{proof}

A \emph{chain} consists of objects $a_i$ and specified arrows $e_i:a_i\to a_{i+1}$. Its internal arrow $e_i^j:a_i\to a_j$ is their composite for $i<j$, and $e_i^i=\id_{a_i}$. The chain is \emph{cofinal} if every object admits an arrow to some $a_i$. It has the \emph{extension property} if
\begin{equation}\label{eq:extension-property}
 \begin{split}
 e:a_i\to b,\quad N\in\N\quad\Longrightarrow\quad
 &\exists j>\max\{i,N\}\ \exists d:b\to a_j\\[-2pt]
 &\hspace{25mm}d\circ e=e_i^j.
 \end{split}
\end{equation}

\begin{lemma}\label{lem:fusion}
We can select Borelly a cofinal sequence of witnessed conditions with
the extension property, unbounded levels, and cofinal root kernels.
An isomorphism $P_R\cong_{\mathrm{top}}P_S$ induces commuting cofinal zigzags between
the selected sequences.
\end{lemma}

\begin{proof}
Starting with the witnessed seed, dovetail the first three requirements
below with every prescribed lower bound on the future chain index; the
fourth requirement independently supplies unbounded word levels.
\begin{enumerate}
\item For every witnessed object $b$, find an arrow $b\to a_j$.
\item For every witnessed arrow $e:a_i\to b$, find $j>i$ and
      $d:b\to a_j$ whose composite with $e$ is the internal arrow
      $a_i\to a_j$.
\item For every $k$ with $R\subseteq M_k$, put a future root kernel
      inside $M_k/R$.
\item For every $N\in\N$, make a future level exceed $N$.
\end{enumerate}
The chain index $i$ and the word level $n(a_i)$ play different roles:
the former orders the selected diagrams, whereas the latter specifies
the finite tables representing them. Every tuple of codes and chain-index
lower bounds is revisited infinitely often, as is every requested word
level.
Invalid codes and references to chain stages not yet defined impose no
action at that visit. Before meeting a valid requirement, insert finitely
many identity arrows so that the current tail index exceeds its
prescribed lower bound.

Write $a_t$ for the current tail. Meet the first requirement by joint
embedding with $a_t$, and the
second by amalgamating $a_i\to b$ with the internal arrow from $a_i$
to the tail. The third follows from Lemma~\ref{lem:amalgamation}.
For the fourth, continue the current condition to
$m>\max\{N,n(a_t)\}$. Its construction record gains a continuation node;
Lemma~\ref{lem:continuation} preserves the regular surjections, their
associated group maps, and all recorded equations. The witness is
unchanged, and the identity on the continued root gives the required
arrow.

Choose the least code of a witnessing diagram at each step. Such codes
exist by these constructions and satisfy Borel predicates. For any code
$v$, being least means that $v$ witnesses the requirement and no smaller
code witnesses the same requirement. This is a Borel condition. Induction makes every selected object and
bonding table Borel. The repeated lower bounds give the extension
property \eqref{eq:extension-property}. The other requirements give
cofinality, cofinal root kernels, and unbounded levels.

Let $\theta:P_R\to P_S$ be a topological group isomorphism. For a witnessed
condition $a$ with root witness $(k,\beta)$, put $U=M_k/R$. Put
$U'=\theta[U]$, let $M_l$ be its inverse
image in $\Fhat$, and transport the witness by
\begin{equation}\label{eq:transported-witness}
 \beta':P_S/U'\longrightarrow G_{r(a)},\qquad
 \beta'(sU')=\beta\bigl(\theta^{-1}(s)U\bigr),
\end{equation}
using $P_R/U=\Fhat/M_k$. Since $M_l$ is the inverse image of $U'$ in
$\Fhat$, we also have $P_S/U'=\Fhat/M_l$; through this identification,
$(l,\beta')$ is a root witness for $a$ over $P_S$. This transport fixes
the finite groups, word systems, construction records, and tables, and preserves
\eqref{eq:witness-compatibility}. Transport by $\theta^{-1}$ is its
inverse, so the witnessed categories are isomorphic.

In the common category, let $(a_i),(b_j)$ be the two selected chains.
Cofinality gives $a_{i_0}\to b_{j_0}$. Apply
\eqref{eq:extension-property} alternately to obtain
\[
 a_{i_0}\longrightarrow b_{j_0}\longrightarrow a_{i_1}
 \longrightarrow b_{j_1}\longrightarrow\cdots,
\]
requiring $i_t,j_t\geq t$ and each two-step composite to equal the
internal arrow. An arrow $a_{i_t}\to b_{j_t}$ gives, contravariantly,
the $i_t$-coordinate of a map from the $b$-limit to the $a$-limit; an
arrow $b_{j_t}\to a_{i_{t+1}}$ gives the corresponding coordinate in
the other direction. The two-step identities make these coordinates
compatible, and cofinality determines all remaining coordinates. The
resulting maps are inverse. Every factor commutes with the shift, so the
resulting homeomorphism is a conjugacy.
\end{proof}

\subsection{A fixed Cantor space}

For $R\in\mathcal N$, let $(a_i)_{i\in\N}$ and
$(e_i:a_i\to a_{i+1})_{i\in\N}$ be the Borel chain and its specified
arrows selected in Lemma~\ref{lem:fusion}. Write $r_i=r(a_i)$ and let
$\pi_i^{i+1}=(p_{e_i})^*:X_{r_{i+1}}\onto X_{r_i}$ be the bonding factors represented by the selected chain arrows.
Set
\begin{equation}\label{eq:dynamical-inverse-limit}
 Y_R=\varprojlim_i(X_{r_i},\pi_i^{i+1}),\qquad
 S_R((x_i)_i)=(\sigma x_i)_i.
\end{equation}
The bonding maps are onto, so $Y_R$ is nonempty and every coordinate
projection is onto. A nonempty basic open set in $Y_R$ pulls back from one sufficiently
late stage. Minimality there gives a finite covering by iterates, and
pullback gives the same covering in $Y_R$. Thus $S_R$ is minimal.
The aligned bonding maps preserve the common odometer coordinate. They
therefore define $\rho_R:Y_R\to O$ by
$\rho_R((x_i)_i)=\rho_{r_i}(x_i)$ for any $i\in\N$; compatibility makes
this value independent of $i$.
Fix a point of $O$. At each stage its fibre is nonempty and compact.
An aligned onto bonding factor restricts
to an onto map between these fibres: every preimage has the same
odometer coordinate. Their inverse limit is nonempty, proving that
$\rho_R$ is onto and hence that $Y_R$ is infinite. An isolated point would, by minimality
and compactness, have finitely many translates covering $Y_R$.
Consequently $Y_R$ is perfect. It is also compact, metrizable, and
zero-dimensional, so it is a Cantor space.

To code these spaces uniformly, fix natural-number symbol codes and a
pairing function. Define
\begin{equation}\label{eq:ambient-embedding}
 \begin{split}
 E_R:Y_R&\longrightarrow(2^{\N})^{\Z},\\
 E_R((x_i)_i)(t)(\langle i,e\rangle)=1
 &\quad\Longleftrightarrow\quad x_i(t)\text{ has code }e,
 \qquad t\in\Z,\quad i,e\in\N.
 \end{split}
\end{equation}
This is a homeomorphism onto its compact image $A_R$ and intertwines
$S_R$ with the ambient shift $\widehat\sigma$.

A basic ambient cylinder constrains finitely many rows and coordinates;
its zero-bit requirements are finite Boolean combinations of row-symbol
conditions.
Pull its conditions to one later selected row through the bonding
tables. For a level-$n$ table, the block containing coordinate zero has
its left endpoint in $[-L_n+1,0]$. Testing these positions in the window
$[-L_n+1,L_n-1]$ identifies it uniquely by readability and determines
the target symbol. Compositions therefore give computable finite local
rules, so the pulled-back condition is finite and clopen. A pattern occurs in
the row subshift exactly when it occurs in two concatenated construction
words at some later level: ordered-pair coverage preserves the pattern
and supplies arbitrarily long extensions. Nonemptiness is therefore a
countable union of finite word tests. Surjectivity of the bonding maps
lifts a satisfying row to $Y_R$. Thus
$\{R:A_R\cap U\neq\varnothing\}$ is Borel for every basic cylinder $U$.
These hit predicates generate the hyperspace Borel structure, proving
that $R\mapsto A_R$ is Borel.

We construct a uniform homeomorphism $\Phi_A:A\to\Cantor$ for each
nonempty perfect compact subset $A$ of the ambient space. Fix a compatible
metric and enumerate the entire countable clopen algebra. Recursively
choose $0=k_0(A)<k_1(A)<\cdots$ and relative clopen partitions
$(P_s(A))_{s\in2^{k_n}}$ refining one another, with mesh (the maximum
cell diameter) below $2^{-n}$
for $n\geq1$, starting with $P_{\varnothing}(A)=A$. At each step choose
the least finite array that partitions
every current cell into a common number $2^r$ of nonempty pieces of the
required mesh, first minimizing $r\geq1$, and set $k_{n+1}=k_n+r$.
Index the terminal pieces by the binary extensions of length $r$ and
define the intervening binary cells by the corresponding unions.
Existence follows by taking
fine clopen partitions and repeatedly splitting their perfect cells to
reach a common power of two. Then
$P_s=P_{s0}\mathbin{\dot\cup}P_{s1}$ at every binary level.

Nonemptiness, relative coverage, disjointness, refinement, and diameter
bounds are Borel tests, so the recursion is Borel. Define
\begin{equation}\label{eq:cantor-trivialization}
 \Phi_A(x)=z\quad\Longleftrightarrow\quad
 x\in P_{z\upharpoonright m}(A)\text{ for every }m.
\end{equation}
Shrinking mesh and nested compactness give bijectivity, and the cells
are precisely the inverse images of binary cylinders. Thus $\Phi_A$
is a homeomorphism. The joint map is Borel by these membership tests;
its inverse is Borel by approximating the unique point in the nested
cells with the first point of a fixed dense sequence in the ambient
space within $2^{-n}$
of the level-$k_n$ cell.
The admissibility of each approximate point is a hyperspace hit test
for an open ball. Hence these choices are Borel, and their distances
to the unique point tend to zero with the mesh.

Set
\begin{equation}\label{eq:fixed-cantor-map}
 T_R=\Phi_{A_R}E_RS_RE_R^{-1}\Phi_{A_R}^{-1}
     =\Phi_{A_R}(\widehat\sigma\!\upharpoonright A_R)\Phi_{A_R}^{-1}.
\end{equation}
Joint evaluation is Borel and every section is continuous. Fix a compatible metric $d$ on $C$, and for continuous maps $f,g:C\to C$ set
\begin{equation}\label{eq:uniform-metric}
 d_\infty(f,g)=\sup_{z\in C}d(f(z),g(z)).
\end{equation}
For $\epsilon>0$, a dense sequence $(z_j)$ in $\Cantor$, and a continuous
map $g:\Cantor\to\Cantor$,
\[
 \{R:d_\infty(T_R,g)<\epsilon\}
 =\bigcup_{\substack{q\in\mathbb Q\cap[0,\infty)\\q<\epsilon}}
   \bigcap_j\{R:d(T_R(z_j),g(z_j))\leq q\}.
\]
The same formula applies to the inverses. Hence $R\mapsto T_R$ is
Borel into $\Homeo(\Cantor)$, with range in $\Min(\Cantor)$.

\section{Recovering the profinite group}\label{sec:recovery}

Lemma~\ref{lem:fusion} and \eqref{eq:fixed-cantor-map} show that
$P_R\cong_{\mathrm{top}}P_S$ implies $T_R\cong T_S$. We prove the converse by
extracting finite quotient homomorphisms from a conjugacy. The first
step uses the finite alphabet of each coordinate subshift.

\begin{lemma}\label{lem:finite-stage}
Let $X=\varprojlim_i X_i$ be an inverse limit of compact zero-dimensional
systems with homeomorphisms and onto bonding maps. Every continuous
equivariant map from $X$ to a finite-alphabet subshift $Z$ factors
equivariantly through one $X_i$. Surjectivity is preserved.
\end{lemma}

\begin{proof}
Let $A_Z$ be the alphabet of $Z$ and $\pi_i:X\onto X_i$ the projections.
The zero-coordinate symbol of $F:X\to Z$ gives a finite clopen
partition of $X$. Compactness expresses its members as finite unions
of sets $\pi_i^{-1}(U)$ with $U\subseteq X_i$ clopen. Pass to one index
above all those used. Since $\pi_i$ is onto, these sets descend to a
clopen partition of $X_i$, defining a continuous rule $c_i:X_i\to A_Z$.
If $T_i$ is the homeomorphism on $X_i$, put
\[
 F_i(u)(n)=c_i(T_i^nu),\qquad n\in\Z.
\]
This map is continuous and equivariant, and $F=F_i\pi_i$.
Surjectivity of $\pi_i$ gives $F_i[X_i]=F[X]\subseteq Z$, proving
both assertions.
\end{proof}

Suppose that $g:C\to C$ conjugates $T_R$ to $T_S$, and set
\[
 h=E_S^{-1}\Phi_{A_S}^{-1}g\Phi_{A_R}E_R:Y_R\longrightarrow Y_S.
\]
Equation~\eqref{eq:fixed-cantor-map} shows that $h$ conjugates $S_R$ to
$S_S$. Write $r_i,s_j$ for the root
indices and $G_i,H_j$ for their groups. Apply
Lemma~\ref{lem:finite-stage} alternately to coordinates of $h$ and
$h^{-1}$: factor the $j_0$-coordinate of $h$ through $i_0$, then the
$i_0$-coordinate of $h^{-1}$ through $j_1>j_0$, then the
$j_1$-coordinate of $h$ through $i_1>i_0$, and continue. Require each
new index to pass the next unused stage. This gives cofinal subsequences
and onto factors
\[
 A_n:X_{r_{i_n}}\to X_{s_{j_n}},\qquad
 B_n:X_{s_{j_{n+1}}}\to X_{r_{i_n}}.
\]
These maps are onto because the corresponding coordinates of $h$ and
$h^{-1}$ are onto and Lemma~\ref{lem:finite-stage} preserves surjectivity.
Writing $\pi^R_{i',i},\pi^S_{j',j}$ for internal bonding factors, we have
\begin{equation}\label{eq:dynamical-triangles}
 B_nA_{n+1}=\pi^R_{i_{n+1},i_n},\qquad
 A_nB_n=\pi^S_{j_{n+1},j_n}.
\end{equation}
Both identities follow from $h^{-1}h=\id$ and $hh^{-1}=\id$ after
composition with the onto inverse-limit projections.

By Theorem~\ref{thm:factors}, write
$A_n=\sigma^{c_n}\widetilde A_n$ and
$B_n=\sigma^{d_n}\widetilde B_n$, with
$\widetilde A_n,\widetilde B_n$ aligned and represented by regular affine
tables. The internal factors
are aligned, so their odometer offsets give
\[
 c_n+d_n=0,\qquad d_n+c_{n+1}=0.
\]
These are integer equalities because $L_m\to\infty$ makes the diagonal
map $\Z\to O$ injective. Hence $c_n=c$, $d_n=-c$, and shift
equivariance reduces \eqref{eq:dynamical-triangles} to
\[
 \widetilde B_n\widetilde A_{n+1}=\pi^R_{i_{n+1},i_n},\qquad
 \widetilde A_n\widetilde B_n=\pi^S_{j_{n+1},j_n}.
\]
At a common word level, these are equalities of evaluated tables by
Theorem~\ref{thm:factors}. Lemma~\ref{lem:affine} gives the corresponding
equalities of associated group homomorphisms. Denote those of
$\widetilde A_n,\widetilde B_n$
by $\alpha_n:G_{i_n}\onto H_{j_n}$ and
$\beta_n:H_{j_{n+1}}\onto G_{i_n}$, respectively. If $q^R,q^S$ are
the witnessed group bonding maps, then
\begin{equation}\label{eq:promaps}
 \beta_n\alpha_{n+1}=q^R_{i_{n+1},i_n},\qquad
 \alpha_n\beta_n=q^S_{j_{n+1},j_n}.
\end{equation}
In particular,
\[
 q^S_{j_{n+1},j_n}\alpha_{n+1}
 =\alpha_n\beta_n\alpha_{n+1}
 =\alpha_nq^R_{i_{n+1},i_n}.
\]
Thus a compatible family $(x_i)\in\varprojlim_iG_i$ determines a unique
point of $\varprojlim_jH_j$ with $j_n$-coordinate
$\alpha_n(x_{i_n})$. This is a continuous group homomorphism.
For the reverse map, given $(y_j)_j\in\varprojlim_jH_j$, prescribe its
$i_n$-coordinate to be $\beta_n(y_{j_{n+1}})$. Compatibility follows from
\[
 q^R_{i_{n+1},i_n}\beta_{n+1}
 =\beta_n\alpha_{n+1}\beta_{n+1}
 =\beta_nq^S_{j_{n+2},j_{n+1}}.
\]
Both maps are therefore well defined; their coordinates on cofinal
subsequences determine all remaining coordinates uniquely.
Equation~\eqref{eq:promaps} shows that they are inverse.

The witnesses identify these limits with $P_R$ and $P_S$. For each selected
condition $a_i$, write $(\kappa_i,\gamma_i)$ for its root witness and set
$N_i=M_{\kappa_i}/R$. The sequence $(N_i)$ is decreasing and cofinal among
the open normal subgroups of $P_R$. The map
\[
 \eta_R:P_R\longrightarrow\varprojlim_iP_R/N_i,
 \qquad x\longmapsto(xN_i)_i
\]
is continuous and injective, since cofinality among the open normal subgroups gives $\bigcap_iN_i=\{1\}$.
A compatible family of cosets is a decreasing
family of nonempty compact sets; its intersection is a singleton by the
same identity. Thus $\eta_R$ is a continuous bijection from a compact
space to a Hausdorff space, hence a homeomorphism. The witnesses $\gamma_i$
identify $P_R/N_i$ with $G_i$, and \eqref{eq:witness-compatibility}
identifies the quotient maps with the group bonding maps. Applying the same argument to $P_S$ proves
$P_R\cong_{\mathrm{top}}P_S$.

Inclusion of $\Min(C)$ into $\Homeo(C)$ reduces minimal conjugacy to
conjugacy of arbitrary Cantor homeomorphisms. Proposition~\ref{prop:countable-universality}
therefore gives the upper bound $\ES$. Conversely, composing the Borel
map of Proposition~\ref{prop:profinite-source} with $R\mapsto T_R$ gives
the lower bound by \eqref{eq:main}. This proves Theorem~\ref{thm:main}.

\section{Limitations}

Combining the reduction with restrictions on invariant measures or word growth requires refinements that retain the quotient maps while controlling word frequencies or word counts. The additional conditions must persist through amalgamation and passage to inverse limits.

\section*{Declaration and Statements}

We leveraged our developed Agent TARS to support this research. GPT‑5.6 Sol assisted with manuscript organization and language polishing. All details were personally checked and validated by Xinan Dai, who also built the core narrative of the article. No generative‑AI tool is listed as an author; all intellectual content has been reviewed and confirmed by human authors.

\end{document}